\documentclass[11pt]{amsart}
  \RequirePackage{amsmath}
\RequirePackage{amssymb} \RequirePackage{amsthm}
\RequirePackage{epsfig} 
\def\ep{\varepsilon}

\newcommand{\D}{\mathbb{D}}

\newcommand{\T}{\mathbb{T}}
\newcommand{\N}{\mathbb{N}}

\newcommand{\C}{\mathbb{C}}

\newcommand{\vp}{\vp}

\def\a{\alpha}               
\def\d{\delta}       \def\De{{\Delta}}    \def\e{\varepsilon}

         \def\r{\rho}         \def\z{\zeta}

\newtheorem{theorem}{Theorem}
\newtheorem{lemma}[theorem]{Lemma}

\newtheorem{lettertheorem}{Theorem}

\newtheorem{letterlemma}[lettertheorem]{Lemma}
\newtheorem{letterproposition}[lettertheorem]{Proposition}

\theoremstyle{definition}

\theoremstyle{remark}

\theoremstyle{remarks}

\numberwithin{equation}{section}
\begin{document}
\title[Compact embedding derivatives]
{Compact embedding derivatives of Hardy spaces into Lebesgue spaces}

\author{Jos\'e \'Angel Pel\'aez}

\address{Departamento de An아lisis Matem아tico, Universidad de M아laga, Campus de
Teatinos, 29071 M아laga, Spain} \email{japelaez@uma.es}

\thanks{ The author was supported in part by the Ram\'on y Cajal program
of MICINN (Spain), Ministerio de Edu\-ca\-ci\'on y Ciencia, Spain,
(MTM2011-25502), from La Junta de Andaluc{\'i}a, (FQM210) and
(P09-FQM-4468).}
\date{\today}



\begin{abstract}
We characterize the positive Borel measures such that the differentiation operator of order $n\in\N\cup\{0\}$ is compact from
the Hardy space $H^p$ into $L^q(\mu)$, $0<p,q<\infty$.
\end{abstract}
\maketitle



\section{Introduction}
Let $\D$ denote the open unit disk of the complex plane and let $\T$
denote the unit circle.
 Also, let $H^p,~0<p<\infty$ denote the
standard Hardy space of analytic functions in $\D$.\par The aim of this paper is to characterize
the positive Borel measures $\mu$ on the unit disc $\D$ such that the differentiation operator
 $D^{(n)}(f)=f^{(n)}$
  is compact from $H^p$ into $L^q(\mu)$, $n\in\N\cup\{0\}$ and $0<p,q<\infty$.

\par The analogous problem for the standard Bergman spaces $A^p_\a$ has been solved \cite{Luecking1985,Lu93,ZhuZhao08}.
The formula $||f||_{A^p_\a}\asymp \sum_{j=0}^{n-1}|f^{(j)}(0)|+||f^{(n)}||_{A^p_{\a+np}}$
 implies that for these spaces the question of when the differentiation operator $D^{(n)}$ from $A^p_\a$ into $L^q(\mu)$ is bounded or compact,
  can be answered once the case $n=0$ is solved. However, this method does not work for Hardy spaces,
   because  such a Littlewood-Paley formula does not exist for $p\neq 2$. Nevertheless,
 an equivalent $H^p$-norm in terms of the $n$th derivative
can be given by using the square functions
$$S_{\sigma,n}f(\zeta)=\left(\int_{\Gamma_\sigma(\zeta)}|f^{(n)}(z)|^2(1-|z|)^{2n-2}\,dA(z)\right)^{1/2}$$
where
 $\Gamma_\sigma(\zeta)$ $=\left\{z\in \D:\,|\arg\z-\arg
    z|<\sigma\left(1-|z|\right)\right\}$ denotes the Stolz angle (lens type region)  with vertex at
$\zeta\in\T$ and aperture $\sigma>0$.
Precisely,  for $0<p<\infty$ \cite{AhernBruna1988,FefSt}
\begin{equation}\label{fsest}\|f\|_{H^p}^p \asymp \sum_{j=0}^{n-1}|f^{(j)}(0)|^p+\int_\T
S^p_{\sigma,n}f(\zeta)\,dm(\zeta)\,.\end{equation} Here and throughout in what
follows $m$ denotes the arclength measure on $\T$.
 In view of \eqref{fsest} and the $H^p$ characterization through the non-tangential maximal function,
    it is natural that
     Luecking~\cite{Lu90,Lu93} employed tent spaces to describe the positive Borel measures such that $D^{(n)}:H^p\to L^q(\mu)$ is bounded.
    It is worth noticing that Coifman, Meyer and Stein~\cite{CMS} introduced the theory of tent spaces in an harmonic analysis context.
    It was further extended by Cohn and Verbitsky~\cite{CV} and has become a very useful tool in operator theory on Hardy spaces
    \cite{CV,CohnFerRochPLMS01}.
\medskip\par Let us now recall some  definitions that will enable to
state 
 the solution to the primary question of this paper.
Let us write $\Gamma_{1/2}(\z)=\Gamma(\z)$ for short, and for each $z\in\D$ let be
$I(z)=\left\{\z\in \T:\,z\in\Gamma(\z)\right\}$
 the related interval.
The Carleson square $S(I)$ based on an
interval $I\subset\T$ is the set $S(I)=\{re^{it}\in\D:\,e^{it}\in I,\,
1-|I|\le r<1\}$, where $|E|$ denotes the Lebesgue measure of $E\subset\T$.

If $|I|<\frac{\pi}{4}$ the tent $T(I)$  is the open subset  of $\D$ bounded by the arc
$I\subset \T$ and two straight lines through
the endpoints of $I$ forming with $I$ an angle of $\frac{\pi}{4}$.
 If $|I|\ge \frac{\pi}{4}$, we set $T(I)=\cup_{J\subset I,\,|J|<1}T(J)\cup\{0\}$. For each  $a\in\D$, let be $S(a)=S(I(a))$ and  $T(a)=T(I(a))$.

For $0<q<\infty$ and a positive Borel measure $\nu$ on $\D$, finite on compact sets, denote
    $
    A^q_{q,\nu}(f)(\z)=\int_{\Gamma(\z)}|f(z)|^q\,d\nu(z)
    $
and $A_{\infty,\nu}(f)(\z)=\nu\textrm{-ess}\sup_{z\in\Gamma(\z)}|f(z)|$. For $0<p<\infty$, $0<q\le\infty$  the tent space $T^p_q(\nu)$ consists of the $\nu$-equivalence classes of $\nu$-measurable functions $f$ such that $\|f\|_{T^p_q(\nu)}=\|A_{q,\nu}(f)\|_{L^p(\T,m)}$ is finite.
 For $0<q<\infty$, define
    $$
    C^q_{q,\nu}(f)(\z)=\sup_{a\in\Gamma(\z)}\frac{1}{|I(a)|}\int_{T(a)}|f(z)|^q (1-|z|)\,d\nu(z),\quad \z\in\T.
    $$
A quasi-norm in the tent space $T^\infty_q(\nu)$ is defined by $\|f\|_{T^\infty_q(\nu)}=\|C_{q,\nu}(f)\|_{L^\infty(\T,m)}$.
\par The following result gives a description of the dual of $T^p_q(\nu)$ \cite{CMS,Lu90}. We also refer to \cite[Theorem $4$]{PelRatMathAnn}
where an analogue was proved for a family of weighted tent spaces on the unit disc.

\begin{lettertheorem}\label{Thm:tent-spaces-duality}
Let $1\le p,q<\infty$ with $p+q\ne2$ and let $\nu$ be a positive Borel measure on~$\D$, finite on compact sets of $\D$. Then the dual of
 $T^p_q(\nu)$ can be identified with $T^{p'}_{q'}(\nu)$ (up to an equivalence of norms) under the pairing
    \begin{equation*}\label{Eq:duality-pairing-tent-spaces}
    \langle f,g \rangle_{T^2_2(\nu)}=\int_\D f(z)\overline{g(z)}(1-|z|)\,d\nu(z).
    \end{equation*}
\end{lettertheorem}

For the sake of completeness, and because it is a key to
 describe  those positive Borel measures such that $D^{(n)}: H^p\to L^q(\mu)$ is compact,
 we shall prove in Section \ref{section2} that each $g\in T^{p'}_{q'}(\nu)$ induces a bounded linear functional on $T^p_q(\nu)$.
In the proof for $p=1$, a stopping time involving $A_{q,\nu}(f)$ and $C_{q,\nu}(f)$
is a fundamental step.  Let $\Delta(a,r)$ and $D(a,r)$ respectively denote the
pseudohyperbolic and Euclidean discs of center $a$ and radius~$r$.
Our main result is the following.
\begin{theorem}\label{Theorem:DifferentiationOperatorcompact}
Let $0<p,q<\infty$,  $n\in\N$ and $\mu$ be a positive Borel measure on~$\D$. Further, let $dh(z)=dA(z)/(1-|z|^2)^2$ denote the hyperbolic measure.
\begin{itemize}
\item[\rm(a)] If $p\ge q$, $D^{(n)}:H^p\to L^q(\mu)$ is compact if and only if, for any fixed $r\in(0,1)$, the function
    $$
    \Phi_\mu(z)=\frac{\mu(\Delta(z,r))}{(1-|z|)^{1+qn}},\quad z\in\D,
    $$
satisfies that
    \begin{enumerate}
    \item[\rm(i)] $\Phi_\mu\in T^\frac{p}{p-q}_{\frac{2}{2-q}}\left(h\right)$, if $q<\min\{2,p\}$;
    \item[\rm(ii)] $\lim_{|a|\to 1^-} \frac{1}{|I(a)|}\int_{T(a)}|\Phi_\mu(z)|^{\frac{2}{2-p}} \,\frac{dA(z)}{ 1-|z|}=0$, if $q=p<2$;
    \item[\rm(iii)] $\lim_{R\to 1^-}\int_\T\left(\sup_{z\in{\Gamma(\z)\setminus\overline{D(0,R)}}}\Phi_\mu(z)\right)^{\frac{p}{p-q}}\,dm(\z)=0$, if $2\le q<p$.
    \end{enumerate}
    \medskip
\item[\rm(b)] If either $q>p$ or $2\le q=p$, the following conditions are equivalent:
\begin{enumerate}
\item[\rm(i)] $D^{(n)}:H^p\to L^q(\mu)$ is compact;
\item[\rm(ii)] $\lim_{|z|\to 1^-}\frac{\mu\left(S(z)\right)}{(1-|z|)^{\frac{q}{p}+nq}}=0$ ;
\item[\rm(iii)] $\lim_{|z|\to 1^-}\frac{\mu\left(\Delta(z,r)\right)}{(1-|z|)^{\frac{q}{p}+nq}}=0$  for any fixed $r\in(0,1)$.
\end{enumerate}
\end{itemize}
\end{theorem}
\par As for $n=0$,  $I_d: H^p\to L^q(\mu)$ is compact if and only if
$\lim_{|z|\to 1^-}\frac{\mu\left(S(z)\right)}{(1-|z|)^{\frac{q}{p}}}=0$, whenever $0<p\le q<\infty$ \cite{BlJa05}. In the previous condition
 $\mu\left(S(z)\right)$ may be replaced by $\mu\left(\Delta(z,r)\right)$ if $p<q$. In the triangular case $0<q<p<\infty$,
  $I_d: H^p\to L^q(\mu)$ is bounded if and only   the function
    $
    B_\mu(\z)=\int_{\Gamma(\z)}\frac{d\mu(z)}{1-|z|}
    $
belongs to $L^{\frac{p}{p-q}}(\T,m)$ \cite{Lu90}. For this last range of values
it is probably known, at least to experts working on the field, that
  $I_d: H^p\to L^q(\mu)$ is compact if and only it is  bounded. Since we were not able to
find a proof in the existing literature, we include a proof here.

\begin{theorem}\label{carlesonq<p}
Let $0<q<p<\infty$
and let $\mu$ be a positive Borel measure on $\D$. Then the following conditions are equivalent:
\begin{enumerate}
\item[\rm(i)] $I_d:H^p\to L^q(\mu)$ is compact;
\item[\rm(ii)] $I_d:H^p\to L^q(\mu)$ is bounded;
\item[\rm(iii)] The function
    $
    B_\mu(\z)=\int_{\Gamma(\z)}\frac{d\mu(z)}{1-|z|}
    $
belongs to $L^{\frac{p}{p-q}}(\T,m)$.
\end{enumerate}
\end{theorem}

Throughout the paper  $\frac{1}{p}+\frac{1}{p'}=1$. Further, the letter $C=C(\cdot)$ will denote an
absolute constant whose value depends on the parameters indicated
in the parenthesis, and may change from one occurrence to another.
We will use the notation $a\lesssim b$ if there exists a constant
$C=C(\cdot)>0$ such that $a\le Cb$, and $a\gtrsim b$ is understood
in an analogous manner. In particular, if $a\lesssim b$ and
$a\gtrsim b$, then we will write $a\asymp b$.

\section{Preliminary  results}\label{section2}

\begin{letterproposition}\label{Prop:tent-spaces-duality}
Let $1\le p,q<\infty$ with $p+q\ne2$ and let $\nu$ be a positive Borel measure on~$\D$, finite on compact sets of $\D$.
Then, there exists a positive constant $C$ such that
    \begin{equation*}
    \left|\langle f,g \rangle_{T^2_2(\nu)}\right|\le C||f||_{T^p_q(\nu)}\|g\|_{T^{p'}_{q'}(\nu)}
    \end{equation*}
    for any $f\in T^p_q(\nu)$ and $g\in T^{p'}_{q'}(\nu)$.
\end{letterproposition}

\begin{proof}
If $1<p,q<\infty$, then Fubini's theorem and two applications of H\"older's inequality give
    \begin{equation}
    \begin{split}\label{p1}
    |\langle f,g \rangle_{T^2_2(\nu)}|
    &\le\int_\T A_{q,\nu}(f)(\z) A_{q',\nu}(g)(\z)\,dm(\z)
   =\|f\|_{T^p_q(\nu)}\|g\|_{T^{p'}_{q'}(\nu)}.
    \end{split}
    \end{equation}
If $q=1$ and $1<p<\infty$, then H\"older's inequality yields
    \begin{equation}
    \begin{split}\label{p2}
    |\langle f,g \rangle_{T^2_2(\nu)}|
    &\le\int_\T A_{1,\nu}(f)(\z)A_{\infty,\nu}(g)(\z)\,dm(\z)
    =\|f\|_{T^p_1(\nu)}\|g\|_{T^{p'}_{\infty}(\nu)}.
    \end{split}
    \end{equation}
Let now $p=1$ and $1<q<\infty$. For $\z\in\T$ and $0\le h\le\infty$, let
    \begin{equation*}
    \begin{split}
    \Gamma^h(\z)&=\Gamma(\z)\setminus \overline{D\left(0,\frac{1}{1+h}\right)}
    =\left\{z\in\D:|\arg z-\arg \z|<\frac{1-\left|z\right|}{2}<\frac{h}{2(1+h)}\right\}
    \end{split}
    \end{equation*}
and
    $$
    A^{q'}_{q',\nu}(g|h)(\z)=\int_{\Gamma^h(\z)}|g(z)|^{q'}\,d\nu(z),\quad \z\in\T.
    $$
For every $g\in T^\infty_{q'}(\nu)$ and $\z\in\T$, define the stopping time by
    $$
    h(\z)=\sup\left\{h:A_{q',\nu}(g|h)(\z)\le C_1C_{q',\nu}(g)(\z)\right\},
    $$
where $C_1>0$ is a large constant to be determined later. Assume for a moment that there exists a constant $C_2>0$ such that
    \begin{equation}\label{Eq:stopping-time}
    \int_\D k(z)(1-|z|)\,d\nu(z)\le C_2\int_\T\left(\int_{\Gamma^{h(\z)}(\z)}k(z)\,d\nu(z)\right)\,dm(\z)
    \end{equation}
for all $\nu$-measurable non-negative functions $k$.  Then, applying H\"older's inequality
    \begin{equation}\label{Eq:CMS-estimate}
    \begin{split}
    |\langle f,g\rangle_{T^2_2(\nu)}|
    &\le C_2\int_\T\left(\int_{\Gamma^{h(\z)}(\z)}|f(z)||g(z)|\,d\nu(z)\right)\,dm(\z)\\
    &\le C_1C_2\int_\T A_{q,\nu}(f)(\z) C_{q',\nu}(g)(\z)
    \,dm(\z)\\
    &\lesssim\|f\|_{T^1_q(\nu)}\|g\|_{T^\infty_{q'}(\nu)}.
    \end{split}
    \end{equation}
Now let us prove \eqref{Eq:stopping-time}. Fubini's theorem yields
    $$
    \int_\T\left(\int_{\Gamma^{h(\z)}(\z)}k(z)\,d\nu(z)\right)\,dm(\z)
    =\int_\D |(I(z)\cap H(z)| k(z)\,d\nu(z),
    $$
where $H(z)=\{\z\in\T:\frac{1}{1+h(\z)}\le |z|\}$, so it suffices to show that
    \begin{equation}\label{Eq:stopping-time-set-estimate}
    \frac{|(I(z)\cap H(z)|}{|I(z)|}\ge\frac{1}{C_2}
    \end{equation}
for all $z\in\D$. We will prove this only for $z$ close enough to the boundary $\T$, the proof for other values of $z$ follows from this reasoning with appropriate modifications.
 For $|z|\ge1-\frac1n$, set $z'=(1-n(1-|z|))z/|z|$ and $x=\frac1{|z|}-1$,  where $n$ is a natural number $\ge 2$  chosen such that
 $I(z)\cap I(u)=\emptyset$ if  $u\notin T(z')\bigcup\left(\D\setminus\overline{D(0,|z|)}\right)$.
This together with Fubini's theorem  gives
    \begin{equation}\label{st1}
    \begin{split}
    &\frac1{|I(z)|}\int_{I(z)}\left(\int_{\Gamma^x(\z)}|g(u)|^{q'}\,d\nu(u)\right)\,dm(\z)\\
    &=\frac1{|I(z)|}\int_{\{|z|<|u|<1\}}|I(z)\cap I(u)||g(u)|^{q'}\,d\nu(u)\\
    &\le\frac1{|I(z)|}\int_{T(z')}|I(z)\cap I(u)||g(u)|^{q'}\,d\nu(u)\\
    &\le \frac{C_3}{|I(z')|}\int_{T(z')}|g(u)|^{q'}(1-|u|)\,d\nu(u)
    \le C_3\inf_{v\in I(z)}C^{q'}_{q',\nu}(g)(v),
    \end{split}
    \end{equation}
where the  last inequality is valid because
    $$
    \frac1{|I(z')|}\int_{T(z')}|g(u)|^{q'}(1-|u|)\,d\nu(u)
    \le \sup_{a\in\Gamma(v)}\frac1{|I(a)|}\int_{T(a)}|g(u)|^{q'}(1-|u|)\,d\nu(u)
    $$
for all $v\in I(z)$. Denote $E(z)=\T\setminus H(z)=\{\z\in\T:(1+h(\z))|z|<1\}$. By the definition of $h(\z)$ and \eqref{st1}, and by choosing $C_1$ sufficiently large so that $C_1^{q'}>2C_3$, we deduce
    \begin{equation*}
    \begin{split}
    | I(z)\cap E(z)|
    &\le\int_{I(z)}\frac{A^{q'}_{q',\nu}(g|x)(\z)}{C_1^{q'}C^{q'}_{q',\nu}(g)(\z)}\,dm(\z)\\
    &\le\frac{1}{C_1^{q'}\inf_{v\in I(z)}C^{q'}_{q',\nu}(g)(v)}\int_{I(z)}A^{q'}_{q',\nu}(g|x)(\z)\,dm(\z)\\
    &\le\frac{C_3 |I(z)|}{C_1^{q'}}<\frac12 |I(z)|.
    \end{split}
    \end{equation*}
Therefore,
    \begin{equation*}
    \begin{split}
    &\ \frac{| I(z)\cap E(z)|}{|I(z)|}
    =1- \frac{|I(z)\cap E(z)|}{|I(z)|}\ge\frac12,\quad |z|\ge 1-\frac1n,
    \end{split}
    \end{equation*}
and the inequality \eqref{Eq:stopping-time-set-estimate} follows.
\end{proof}
The  reverse implication of Theorem \ref{Thm:tent-spaces-duality} \cite{Lu93} can be proved by using geometric ideas,
 the boundedness of maximal functions and interpolation theorems on $L^pL^q(\nu,m)$ due
to Benedek and Panzone \cite{BenedekPanzone}.

\par The analogue of the following result on $\mathbb{R}^n\times(0,\infty)$  was proved in \cite[Proposition $1$]{Lu93}. See also \cite[Lemma $4$]{PelRatMathAnn}.
\begin{letterlemma}\label{Lemma:cone-integral-nu}
Let $0<p<\infty$ and let $\nu$ be a positive Borel measure on~$\D$, finite on compact sets. Then there exists $\lambda_0=\lambda_0(p)\ge1$ such that
    \begin{equation}\label{Eq:cone-integral-nu}
    \int_\T\left(\int_\D\left(\frac{1-|z|}{|1-\overline{\z}z|}\right)^\lambda d\nu(z)\right)^p\,dm(\z)\asymp\int_\T\left(\nu(\Gamma(\z))\right)^p\,dm(\z)+\nu(\{0\})
    \end{equation}
for each $\lambda>\lambda_0$.
\end{letterlemma}

We defined the tent space $T^p_q(\nu)$ by using the lenses $\Gamma(\z)$.
   Different types of non-tangential approach regions could be used and they would induce the same spaces.
In particular,
the proof of Lemma~\ref{Lemma:cone-integral-nu} shows that we may replace $\Gamma(\z)$ by $\Gamma_\alpha(\z)$ for any $\alpha\in(0,\pi)$ in \eqref{Eq:cone-integral-nu}, and consequently  the space $T^p_q(\nu)$ is independent of the aperture of the lens appearing in the definition,
and the quasi-norms obtained for different lenses are equivalent.

Recall that $Z=\{z_k\}_{k=0}^\infty\subset\D$ is called a
separated sequence if it
is separated in the pseudohyperbolic metric, it is an $\e$-net if $\D=\bigcup_{k=0}^\infty \Delta(z_k,\e)$, and finally
 it is a
$\delta$-lattice if it is a $5\delta$-net and separated with constant $\gamma=\delta/5$.
If $\nu=\sum_k \d_{z_k}$,  then we write $T^p_q(\nu)=T^p_q(\{z_k\})$.
The next result \cite[Theorem $2$]{Lu93} (see also \cite[Lemma $6$]{PelRatMathAnn}) plays an essential role in the proof of Theorem~\ref{Theorem:DifferentiationOperatorcompact} (a).
\begin{letterlemma}\label{Lemma:operator-tent-bergman}
Let $0<p<\infty$ and  let $\{z_k\}$ be a separated sequence. Define
    $$
    S_\lambda(f)(z)=\sum_k f(z_k)\left(\frac{1-|z_k|}{1-\overline{z}_kz}\right)^\lambda,\quad z\in\D.
    $$
Then $S_\lambda:T^p_2(\{z_k\})\to H^p$ is bounded for all $\lambda>\lambda_0$, where $\lambda_0=\lambda_0(p)\ge1$ is that of Lemma~\ref{Lemma:cone-integral-nu}.
\end{letterlemma}

\par We shall also use the following inequality.
Here and on the following  $\triangle$ denotes the Laplacian.
\begin{lemma} If $q\ge 2$ and $0<r<1$ there is a constant $C(q,r)>0$ such that
\begin{equation}\label{Eq:Laplacian-first-derivative}
    \begin{split}
    |f'(z)|^q(1-|z|^2)^q\le C(q,r) \int_{\Delta(z,r)}\triangle|f|^q(\z)\,dA(\z),\quad z\in\D.
    \end{split}
    \end{equation}
\end{lemma}
\begin{proof}
Let $r\in(0,1)$ be fixed. The  classical Hardy-Stein-Spencer identity
$\|f\|_{H^q}^q=|f(0)|^q+
\frac{1}{2}\int_{\D}\triangle|f(z)|^q\,\log\frac{1}{|z|}\,dA(z)\,
$ and
the fact that the Laplacian $\triangle|f|^q$ is subharmonic when $q\ge2$ give
    \begin{equation*}
    \begin{split}
    |f'(0)|^q&\le\left(\|f\|_{H^2}^2-|f(0)|^2\right)^\frac
    q2\le\|f\|_{H^2}^q-|f(0)|^q\le\|f\|_{H^q}^q-|f(0)|^q\\
    &=\frac{1}{2}\int_{\D}\triangle|f|^q(z)\log\frac{1}{|z|}\,dA(z)\le C(q)\int_{\D}\triangle|f|^q(z)(1-|z|)\,dA(z).
    \end{split}
    \end{equation*}
An application of this inequality to the function $f(rz)$ gives
    $$
    |f'(0)|^q\le C(q,r)\int_{\De(0,r)}\triangle|f|^q(z)\left(1-\frac{|z|}{r}\right)\,dA(z).
    $$
Replace now $f$ by $f\circ\varphi_z$ to obtain
\begin{equation*}
    \begin{split}
    |f'(z)|^q(1-|z|^2)^q&\le C(q,r) \int_{\Delta(z,r)}\triangle|f|^q(\z)\left(1-\frac{|\varphi_z(\z)|}{r}\right)\,dA(\z)\\
    &\le C(q,r)\int_{\Delta(z,r)}\triangle|f|^q(\z)\,dA(\z),\quad z\in\D.
    \end{split}
    \end{equation*}
\end{proof}

\section{Proof of main results}
We begin with proving
 Theorem~\ref{Theorem:DifferentiationOperatorcompact}(a).
\begin{theorem}
Let $0<q\le p<\infty$, $n\in\N$ and let  $\mu$ be a positive Borel measure on $\D$. Then $D^{(n)}:H^p \to L^q(\mu)$ is compact if and only if, for any fixed $r\in(0,1)$, the function
    $$
    \Phi_\mu(z)=\frac{\mu(\Delta(z,r))}{(1-|z|)^{1+qn}}
    $$
    satisfies that
    \begin{enumerate}
    \item[\rm(i)] $\Phi_\mu \in T^\frac{p}{p-q}_{\frac{2}{2-q}}\left(h\right)$, if $q<\min\{2,p\}$;
    \item[\rm(ii)] $\lim_{|a|\to 1^-} \frac{1}{|I(a)|}\int_{T(a)}|\Phi_\mu(z)|^{\frac{2}{2-p}}\,\frac{dA(z)}{ 1-|z|}=0$, if $q=p<2$;
    \item[\rm(iii)]  $\lim_{R\to 1^-}\int_\T\left(\sup_{z\in{\Gamma(\z)\setminus\overline{D(0,R)}}}\Phi_\mu(z)\right)^{\frac{p}{p-q}}\,dm(\z)=0$, if $2\le q<p$.
    \end{enumerate}
\end{theorem}

\begin{proof}
Recall the known estimate \cite[Lemma~2.1]{Luecking1985}
    \begin{equation}\label{Eq:suharmonic-n-derivatives}
    |f^{(n)}(z)|^s\lesssim\frac{1}{(1-|z|)^{2+ns}}\int_{\Delta(z,r)}|f(\z)|^s\,dA(\z),\quad z\in\D,\,s>0.
    \end{equation}
Then, the above inequality  and Fubini's theorem give
    \begin{equation}\label{eq:sufdual}
    \begin{split}
    \|f^{(n)}\|_{L^q(\mu)}^q
    &\lesssim\int_\D\frac{1}{(1-|z|)^{2+(n-1)q}}\int_{\Delta(z,r)}|f'(w)|^q\,dA(w)\,d\mu(z)\\
    &\asymp\int_\D |f'(w)|^q\frac{\mu(\Delta(w,r))}{(1-|w|)^{3+(n-1)q}}|I(w)|\,dA(w)
    \\ &= \int_{\D} \left[|f'(w)|(1-|w|)\right]^q \Phi_\mu(w) |I(w)|\,dh(w)
    \end{split}
    \end{equation}
    Let $\{f_k\}_{k=1}^\infty$ such that  $\sup_{k}||f_k||_{H^p}<\infty$. Then, there is subsequence $\{f_{n_{k}}\}_{k=1}^\infty$ which converges uniformly on
compact subsets of $\D$ to an analytic function  $f$. Let denote $g_k=f_{n_k}-f$, $G_k(w)=|g_k'(w)|^q(1-|w|)^q$  and  $dh_R=dh\chi_{\{R<|z|<1\}}$, $0\le R<1$.
Now, we shall show  that conditions (i)-(iii) are sufficient.
\par{\textbf{(i).}}
Fix $\ep>0$.
Since $\Phi_\mu \in T^\frac{p}{p-q}_{\frac{2}{2-q}}\left(h\right)$, by the dominated convergence theorem there is $R_0$ such that
\begin{equation*}\label{c1}
\sup_{R\ge R_0}||\Phi_\mu||_{T^\frac{p}{p-q}_{\frac{2}{2-q}}\left(h_R\right)}<\ep^q.
\end{equation*}
Next, choose $k_0$ with $|g_k(z)|<\ep$ for any $k\ge k_0$ and $|z|\le R_0$.
Then, bearing in mind \eqref{eq:sufdual} and \eqref{p1} and the inequality $\|G_k\|_{T^{\frac{p}{q}}_{\frac{2}{q}}(h)}\lesssim\|g_k\|^q_{H^p}$
(see \eqref{fsest})
\begin{equation*}
    \begin{split}
    \|g_k^{(n)}\|_{L^q(\mu)}^q
    & \lesssim
    \ep^q\int_{|w|\le R_0} (1-|w|)^q \Phi_\mu(w) |I(w)|\,dh(w)
   \\ & + \int_{\D} G_k(w) \Phi_\mu(w) |I(w)|\,dh_{R_0}(w)
    \\ & = \ep^q\langle (1-|w|)^q ,\Phi_\mu\rangle_{T^2_2(h)}+\langle G_k ,\Phi_\mu\rangle_{T^2_2(h_{R_0})}
   \\ & \le
    \ep^q\|(1-|w)^q\|_{T^{\frac{p}{q}}_{\frac{2}{q}}(h)}
    \|\Phi_\mu\|_{T^{\left(\frac{p}{q}\right)'}_{\left(\frac{2}{q}\right)'}(h)}+
    \|G_k\|_{T^{\frac{p}{q}}_{\frac{2}{q}}(h)}
     \|\Phi_\mu\|_{T^{\left(\frac{p}{q}\right)'}_{\left(\frac{2}{q}\right)'}(h_{R_0})}
     \\ & \lesssim \ep^q\left(\|\Phi_\mu\|_{T^{\left(\frac{p}{q}\right)'}_{\left(\frac{2}{q}\right)'}(h)}+||g_k||^q_{H^p}\right)
      \lesssim \ep^q,
    \end{split}
    \end{equation*}
    So $D^{(n)}:H^p \to L^q(\mu)$ is compact.
This together with
\cite[Theorem $1$(i)]{Lu90} proves~(i).
\par{\textbf{(ii).}}    An standard argument
(see \cite[Theorem $3.4$]{BlJa05} for details)
gives that
 $\lim_{|a|\to 1^-} \frac{1}{|I(a)|}\int_{T(a)}|\Phi_\mu(z)|^{\left(\frac{2}{p}\right)'} (1-|z|)\,dh(z)=0$
if and only if
\begin{equation*}\begin{split}
&\lim_{R\to 1^-} \sup_{a\in\D}\frac{1}{|I(a)|}\int_{T(a)}|\Phi_\mu(z)|^{\left(\frac{2}{p}\right)'} (1-|z|)\,dh_R(z)
 =\lim_{R\to 1^-} \|\Phi_\mu\|_{T^\infty_{\left(\frac{2}{p}\right)'}(h_{R})} =0.
\end{split}\end{equation*}
So fixed $\ep>0$, there is $R_0$ such that

\begin{equation*}\begin{split}\label{R0}
&\sup_{a\in\D, R\ge R_0}\frac{1}{|I(a)|}\int_{T(a)}|\Phi_\mu(z)|^{\left(\frac{2}{p}\right)'} (1-|z|)\,dh_R(z)
 = \sup_{ R\ge R_0}\|\Phi_\mu\|_{T^\infty_{\left(\frac{2}{p}\right)'}(h_{R})}
<\ep^p.
\end{split}\end{equation*}
Let $k_0$ be such that  $\sup_{k\ge k_0, |z|\le R_0}|g_k(z)|<\ep$.
\, Then, by \eqref{eq:sufdual}, \eqref{Eq:CMS-estimate} and \eqref{fsest}
\begin{equation*}
    \begin{split}
    \|g_k^{(n)}\|_{L^p(\mu)}^p &
      \lesssim \ep^p \langle (1-|w|)^p ,\Phi_\mu\rangle_{T^2_2(h)}+\langle G_k ,\Phi_\mu\rangle_{T^2_2(h_{R_0})}
  \\ &  \lesssim \e^p \|(1-|w)^p\|_{T^1_{\frac{2}{p}}(h)}
    \|\Phi_\mu\|_{T^\infty_{\left(\frac{2}{p}\right)'}(h)}+\|G_k\|_{T^1_{\frac{2}{p}}}(h)
    \|\Phi_\mu\|_{T^\infty_{\left(\frac{2}{p}\right)'}(h_{R_0})}
    \\ &  \lesssim \e^p( \|\Phi_\mu\|_{T^\infty_{\left(\frac{2}{p}\right)'}(h)}+||g_k||^p_{H^p})
     \lesssim \e^p,
    \end{split}
    \end{equation*}
    which implies that $D^{(n)}: H^p\to L^p(\mu)$ is compact.
\par{\textbf{(iii).}}
Let us observe that \eqref{Eq:Laplacian-first-derivative} and Fubini's theorem give
   \begin{equation*}
    \begin{split}
    \||G_k\|^{\frac{q}{p}}_{T^{\frac{p}{q}}_{1}(h)}&=\int_\T\left(\int_{\Gamma(\z)}|g'_k(w)|^q(1-|w|)^q\,dh(w)\right)^{\frac{p}{q}}\,dm(\z)\\
    &\lesssim\int_\T\left(\int_{\Gamma(\z)}\int_{\Delta(w,r)}\triangle|g_k|^q(z)\,dA(z)\,dh(w)\right)^{\frac{p}{q}}\,dm(\z)\\
    &\le \int_\T\left(\int_{\Gamma'(\z)}\triangle|g_k|^q(z)\,dA(z)\right)^{\frac{p}{q}} \,dm(\z)
    \end{split}
    \end{equation*}
where $\Gamma'(\z)=\{z:\Gamma(\z)\cap\Delta(z,r)\ne\emptyset\}$. Using Lemma~\ref{Lemma:cone-integral-nu} and a result by Calder\'on~\cite[Theorem~1.3]{Pavlovic2013}, we get
$G_k\in T^{\frac{p}{q}}_{1}(h)$ with
$\|G_k\|_{T^{\frac{p}{q}}_{1}(h)}\lesssim\|g_k\|^q_{H^p}$.
From now on, the proof is analogous to both previous cases, so it will be omitted.

\medskip\par Reciprocally, assume that   $D^{(n)}: H^p\to L^q(\mu)$ is compact.
Let $\{z_k\}$ be a $\d$-lattice such that $z_k\ne0$ for all $k$ and let $\mathcal{C}T^p_2(\{z_k\})=\{f\in T^p_2(\{z_k\}):
||f||_{T^p_2(\{z_k\})}=1\}$. For each  $R\in[ 0,1)$ and $\lambda> \lambda_0$ ($\lambda_0$ is that of Lemma~\ref{Lemma:operator-tent-bergman})
consider the operator
    \begin{equation*}
    S_{\lambda,\,R}(f)(z)=\sum_{|z_k|\ge R } f(z_k)\left(\frac{1-|z_k|}{1-\overline{z}_kz}\right)^\lambda,\quad z\in\D.
    \end{equation*}
    Let us observe that $S_{\lambda,\,0}(f)=S_{\lambda}(f)$. By
    Lemma~\ref{Lemma:operator-tent-bergman}, there exists $C>0$ such that
    $$||S_{\lambda,\,R}(f)||_{H^p}\le C||f||_{T^p_2(\{z_k\})},\quad \text{for each $R\in[ 0,1)$ }. $$
    So by the assumption the closure of the set
$ \left\{D^{(n)}\circ S_{\lambda,\,R}\left(\mathcal{C}T^p_2(\{z_k\})\right)\right\}_{R\in[ 0,1)}$ is compact in $L^q(\mu)$.
So, fixed $\ep>0$, standard arguments assert that there is $\rho$ such
that
\begin{equation}\label{dn1}
\int_{\rho<|z|<1}|D^{(n)}\circ S_{\lambda,\,R}(f)(z)|^q\,d\mu(z)< \ep^q\quad\text{for any $R\in[0,1)$ and $f\in\mathcal{C}T^p_2(\{z_k\})$}.
\end{equation}

 Since $\{z_k\}$ is separated and $\lambda>1$, there is $R_0$ such that for any $R\ge R_0$
$ \sum_{|z_k|\ge R } (1-|z_k|)^{\lambda}<\ep^2$. Joining this with
 Lemma \ref{Lemma:cone-integral-nu}, we get
\begin{equation*}\begin{split}
|D^{(n)}\circ S_{\lambda,\,R}(f)(z)| & \le C (\rho,n)\sum_{|z_k|\ge R } |f(z_k)|(1-|z_k|)^{\lambda}
\\ & \le C_\rho \left( \sum_{|z_k|\ge R } |f(z_k)|^2(1-|z_k|)^{\lambda}\right)^{\frac12}\ep
\\ & \le C (\rho,n) \ep\inf_{\z\in\T}\left( \sum_k |f(z_k)|^2\left|\frac{1-|z_k|}{1-\overline{z}_k\z}\right|^\lambda\right)^{\frac12}
\\ & \le C (\rho,n)\ep\int_\T\left(\sum_k|f(z_k)|^2\left(\frac{1-|z_k|}{|1-\overline{z}_k\z|}\right)^\lambda\right)^\frac{p}{2}\,dm(\z)
\\ & \le C (\rho,n) \ep\|f\|_{T^p_2(\{z_k\})}^p,\quad\text{for any  $|z|\le \rho$ and $R\ge R_0$,}
\end{split}\end{equation*}
which together with \eqref{dn1} gives that
\begin{equation}\label{dn2}
||D^{(n)}\circ S_{\lambda,\,R}(f)||_{L^q(\mu)}\lesssim \ep\|f\|_{T^p_2(\{z_k\})},\quad\text{for all $R\ge R_0$ and  $f\in T^p_2(\{z_k\})$}.
\end{equation}
That is
    \begin{equation*}
    \int_\D\left|\sum_{|z_k|\ge R}f(z_k)\frac{(1-|z_k|)^{\lambda}}{(1-\overline{z}_kz)^{\lambda+n}}\right|^q\,d\mu(z)
    \lesssim \ep^q \|f\|^p_{T^q_2(\{z_k\})},
    \end{equation*}
for all $R\ge R_0$ and  $f\in T^p_2(\{z_k\})$.
Replace now $f(z_k)$ by $f(z_k)r_k(t)$, where $r_k$ denotes the $k$th Rademacher function, and integrate with respect to $t$ to obtain
    $$
    \int_0^1\int_\D\left|\sum_{|z_k|\ge R}f(z_k)\frac{(1-|z_k|)^{\lambda}}{(1-\overline{z}_kz)^{\lambda+n}}r_k(t)\right|^qd\mu(z)\,dt\lesssim
   \ep^q \|f\|^q_{T^p_2(\{z_k\})},
    $$
from which Fubini's theorem and an application of Khinchine's inequality yield
    \begin{equation*}
    I=\int_\D\left(\sum_{|z_k|\ge R}|f(z_k)|^2\frac{(1-|z_k|)^{2\lambda}}{|1-\overline{z}_kz|^{2\lambda+2n}}\right)^\frac{q}{2}\,d\mu(z)\lesssim\ep^q \|f\|^q_{T^p_2(\{z_k\})}.
    \end{equation*}
Now, for any fixed $r\in(0,1)$,
    \begin{equation*}
    \begin{split}
    I&\gtrsim\sum_{|z_j|\ge R}\int_{\Delta(z_j,r)}\left(\sum_{|z_k|\ge R}
    |f(z_k)|^2\frac{(1-|z_k|)^{2\lambda}}{|1-\overline{z}_kz|^{2\lambda+2n}}\right)^\frac{q}{2}\,d\mu(z)\\
    &\gtrsim\sum_{|z_j|\ge R}\left(|f(z_j)|^2\frac{(1-|z_j|)^{2\lambda}}{(1-|z_j|)^{2\lambda+2n}}\right)^\frac{q}{2}\mu(\Delta(z_j,r))\\
    &=\sum_{|z_j|\ge R}\frac{|f(z_j)|^q}{(1-|z_j|)^{qn}}\mu(\Delta(z_j,r)),
    \end{split}
    \end{equation*}
    and hence
    \begin{equation}\label{pippeli}
    \begin{split}
    &\sum_{|z_j|\ge R}\frac{|f(z_j)|^q}{(1-|z_j|)^{qn}}\mu(\Delta(z_j,r))\\
    &=\sum_{|z_j|\ge R}|f(z_j)|^q\left(\frac{\mu(\Delta(z_j,r))}{(1-|z_j|)^{1+qn}}\right)(1-|z_j|)\lesssim
    \ep^q \|f\|_{T^p_2(\{z_k\})}^q\\
    &= \ep^q\left(\int_\T\left(\sum_{z_k\in\Gamma(\z)}|f(z_k)|^2\right)^\frac{p}{2}\,dm(\z)\right)^\frac{q}{p}\\
    &= \ep^q \left(\int_\T \left(\left(\sum_{z_k\in\Gamma(\z)}\left(|f(z_k)|^q\right)^\frac{2}{q}\right)^\frac{q}{2}\right)^\frac{p}{q}\,dm(\z) \right)^\frac{q}{p}.
    \end{split}
    \end{equation}
(ii)\, If $q=p<2$, then $s=\frac{p}{q}=1$ and   $v=2/p>1$, so by Theorem~\ref{Thm:tent-spaces-duality}
$\left(T^1_v(\{z_k\})\right)^\star\simeq T^\infty_{v'}(\{z_k\})$  with equivalence of norms.  Therefore \eqref{pippeli} yields
    $$
    \sup_{a\in\D}\frac1{|I(a)|}\sum_{z_k\in T(a), |z_k|\ge R}\left(\frac{\mu(\Delta(z_k,r))}{(1-|z_k|)^{1+pn}}\right)^\frac{2}{2-p}|I(z_k)|\lesssim \ep^p
    $$
    for all $R\ge R_0$. The above inequality is a discrete version of,
    \begin{equation*}\begin{split}
\sup_{a\in\D, R\ge R_0}\frac{1}{|I(a)|}\int_{T(a)\cap \{R<|z|<1\}}|\Phi_\mu(z)|^{\left(\frac{2}{p}\right)'} \,\frac{dA(z)}{1-|z|}
\lesssim \ep^p.
\end{split}\end{equation*}
So, $\lim_{R\to 1^-}\left( \sup_{a\in\D}\frac{1}{|I(a)|}\int_{T(a)\cap \{R<|z|<1\}}|\Phi_\mu(z)|^{\left(\frac{2}{p}\right)'} \,\frac{dA(z)}{1-|z|}\right)=0$,
which is equivalent to $\lim_{|a|\to 1^-} \frac{1}{|I(a)|}\int_{T(a)}|\Phi_\mu(z)|^{\left(\frac{2}{p}\right)'} \,\frac{dA(z)}{1-|z|}=0.$

(iii) If $2<q<p$, then $s=\frac{p}{q}>1$ and $v=\frac{2}{q}<1$, and hence
\cite[Proposition $3$]{Lu90}
 yields
    $$
    \int_\T\left(\sup_{z_k\in\Gamma(\z)\setminus\overline{D(0,R)}}\frac{\mu(\Delta(z_k,r))}{(1-|z_k|)^{1+qn}}\right)^\frac{p}{p-q}\,dm(\z)<\ep^q,\quad
    \text{for all $R\ge R_0$.}
    $$
from which the assertion follows. The case $q=2$ is proved similarly by using Theorem~\ref{Thm:tent-spaces-duality} instead of
\cite[Proposition $3$]{Lu90}. This finishes the proof.
\end{proof}

Now we deal with the second part of Theorem~\ref{Theorem:DifferentiationOperatorcompact}(b).
\begin{theorem}
Let either $0<p<q<\infty$ or $2\le p=q<\infty$ and $n\in\N$,
and let $\mu$ be a positive Borel measure on $\D$. Then the following conditions are equivalent:
\begin{enumerate}
\item[\rm(i)] $D^{(n)}:H^p\to L^q(\mu)$ is compact;
\item[\rm(ii)] $\lim_{|z|\to 1^-}\frac{\mu\left(S(z)\right)}{(1-|z|)^{nq+\frac{q}{p}}}=0$;
\item[\rm(iii)] $\lim_{|z|\to 1^-}\frac{\mu\left(\Delta(z,r)\right)}{(1-|z|)^{nq+\frac{q}{p}}}=0$ for any fixed $r\in(0,1)$.
\end{enumerate}
\end{theorem}
\begin{proof}
A proof of (i)$\Rightarrow$(ii) (and (i)$\Rightarrow$(iii)) can be obtained  by using the test functions $
    f_{a}(z)=\left(\frac{1-|a|^2}{(1-\overline{a}z)^2}\right)^{1/p},\, a\in\D,
    $
    and a regular reasoning. So it is omitted. \par It is enough to prove (iii)$\Rightarrow$(i).
We shall split the proof in two  cases.
\par{\textbf{Case $\mathbf{0<p<q<\infty}$}.}
The argument follows ideas from the proof of \cite[Theorem~3.1]{Luecking1985}. Let $r\in(0,1)$ be fixed, choose $s\in(p,q)$ and denote
    $$
    d\mu^\star(\z)=(1-|\z|)^{\frac{s}{p}-2}\,dA(\z),\quad \z\in\D.
    $$
 Let $\{f_k\}_{k=1}^\infty$ such that $\sup_k||f_k||_{H^p}<\infty$. Then, there is subsequence $\{f_{n_{k}}\}_{k=1}^\infty$ which converges uniformly on
compact subsets of $\D$ to  to an analytic function  $f$. Let denote $g_k=f_{n_k}-f$.
Fix $\ep>0$, by hypothesis there is $\rho$ such that
 $$\frac{\mu\left(\Delta(z,r)\right)}{(1-|z|)^{nq+\frac{q}{p}}}<\ep^q,\quad\text{if $\rho<|z|<1$}.$$
On the other hand, there is $k_0$ such that  $|g_k(z)|<\ep$ for any $k\ge k_0$ and $|z|\le \rho$.
So, bearing in mind \eqref{Eq:suharmonic-n-derivatives} and Minkowski's inequality in continuous form
\begin{equation*}
    \begin{split}
    \|g_k^{(n)}\|_{L^q(\mu)}^q&\lesssim\int_\D\left(\frac{1}{(1-|z|)^{2+ns}}\int_{\Delta(z,r)}|g_k(\z)|^s\,dA(\z)\right)^\frac{q}{s}\,d\mu(z)\\
    &\lesssim\left(\int_\D|g_k(\z)|^s\frac{(\mu(\Delta(\z,r)))^\frac{s}{q}}{(1-|\z|)^{ns}}\,dh(\z)\right)^\frac{q}{s}\\
    & \lesssim \ep^q \left(\int_{|\z|\le\rho}\frac{(\mu(\Delta(\z,r)))^\frac{s}{q}}{(1-|\z|)^{ns}}\,dh(\z)\right)^\frac{q}{s}
    + \ep^q\left(\int_{\rho<|\z|<1}|g_k(\z)|^s\,d\mu^\star(\z)\right)^\frac{q}{s}
    \\ & \lesssim \ep^q\left( \mu^\star(\D)^\frac{q}{s}
    + ||g_k||^q_{L^s(\mu^\star)}\right).
    \end{split}
    \end{equation*}
  Next, since  $\mu^\star(S(a))\lesssim(1-|a|)^\frac{s}{p}$ for $a\in\D$, by \cite[Theorem $9.4$]{Duren1970} $||g_k||_{L^s(\mu^\star)}\lesssim
  ||g_k||_{H^p}$, which together with the above inequalities implies that $D^{(n)}:H^p\to L^q(\mu)$ is compact.
 \par{\textbf{Case $\mathbf{q=p\ge 2}$}.}
 By \eqref{Eq:suharmonic-n-derivatives}, \eqref{Eq:Laplacian-first-derivative} and Fubini's theorem
 \begin{equation*}
    \begin{split}
    \|f^{(n)}\|_{L^p(\mu)}^p
    &\lesssim\int_\D\frac{1}{(1-|z|)^{2+(n-1)p}}\left(\int_{\Delta(z,\r)}|f'(\z)|^p\,dA(\z)\right)\,d\mu(z)\\
    &\lesssim\int_\D\frac{1}{(1-|z|)^{2+np}}\left(\int_{\Delta(z,\r)}\left(\int_{\Delta(\z,s)}\triangle|f|^p(u)\,dA(u)\right)\,dA(\z)\right)\,d\mu(z)\\
    &\lesssim\int_\D\frac{1}{(1-|z|)^{np}}\left(\int_{\Delta(z,r)}\triangle|f|^p(u)\,dA(u)\right)\,d\mu(z)\\
    &\asymp\int_\D\frac{\triangle|f|^p(u)}{(1-|u|)^{np}}\mu(\Delta(u,r))\,dA(u),
    \end{split}
    \end{equation*}
where $\r,s\in(0,1)$ are chosen sufficiently small depending only on $r$. Putting together this inequality with
the  Hardy-Stein-Spencer identity
$\|f\|_{H^p}^p=|f(0)|^p+
\frac{1}{2}\int_{\D}\triangle|f(z)|^p\,\log\frac{1}{|z|}\,dA(z)\,
$,  the proof
 can be finished as in the previous case.

\end{proof}

 The following result contains Theorem~\ref{carlesonq<p}.
\begin{theorem}\label{carlesonq<pextended}
Let $0<q<p<\infty$
and let $\mu$ be a positive Borel measure on $\D$. Then the following conditions are equivalent:
\begin{enumerate}
\item[\rm(i)] $I_d:H^p\to L^q(\mu)$ is compact;
\item[\rm(ii)] $I_d:H^p\to L^q(\mu)$ is bounded;
\item[\rm(iii)] The function
    $$
    \displaystyle \Psi_\mu(\z)=\int_{\D}\left(\frac{1-|z|}{|1-\overline{\z}z|}\right)^{\lambda}\frac{d\mu(z)}{1-|z|}
    $$
belongs to $L^{\frac{p}{p-q}}(\T,m)$ for all $\lambda>0$ large enough;
 \item[\rm(iv)] The function
    $
    B_\mu(\z)=\int_{\Gamma(\z)}\frac{d\mu(z)}{1-|z|}
    $
belongs to $L^{\frac{p}{p-q}}(\T,m)$;

 \item[\rm(v)] For each $0<r<1$, the function
    $$
    \z\mapsto\int_{\Gamma(\z)}\frac{\mu\left(\Delta(z,r)\right)}{(1-|z|)^3}\,dA(z)
    $$
belongs to $L^{\frac{p}{p-q}}(\T,m)$;
\item[\rm(vi)] $M(\mu)(z)=\sup_{z\in S(a)}\frac{\mu(S(a))}{1-|a|}\in L^{\frac{p}{p-q}}(\T,m)$.
\end{enumerate}
\end{theorem}
\begin{proof}
The equivalences $(ii)\Leftrightarrow(iii)\Leftrightarrow(iv) \Leftrightarrow(vi)$ follows from \cite[Section D]{Lu90}.
Next, fixed $0<r<1$, let us observe that for a positive measure $\nu$
\begin{equation}\begin{split}\label{s3}
\int_\D \left(\frac{1-|z|}{|1-\overline{\z}z|}\right)^{\lambda} d\nu(z)
     & =\int_\D\left(\int_{\Delta(u,r)}\left(\frac{1-|z|}{|1-\overline{\z}z|}\right)^{\lambda}\frac{1}{|\Delta(z,r)|}  d\nu(z)\right) dA(u)
    \\ & \asymp   \int_\D \left(\frac{1-|u|}{|1-\overline{\z}u|}\right)^{\lambda}\nu\left(\Delta(u,r)\right)\,dh(u),\quad\text{for all $\z\in\T$.}
    \end{split}
    \end{equation}
    So choosing $d\nu(z)=\frac{d\mu(z)}{1-|z|}$ and applying Lemma~\ref{Lemma:cone-integral-nu}, we have that $(iv)\Leftrightarrow (v)$.
\par Finally, let us see $(iv)\Rightarrow(i)$. By hypothesis and \eqref{s3}
$$
\int_\T \left(\int_\D \left(\frac{1-|u|}{|1-\overline{\z}u|}\right)^{\lambda}\frac{\mu\left(\Delta(u,r)\right)}{(1-|u|)^3}\,dA(u)
\right)^{\frac{p}{p-q}}\,dm(\z)<\infty.$$
Then,  by dominated convergence theorem  and  \eqref{s3},
\begin{equation*}\begin{split}\label{s4}
0 &=\lim_{R\to 1^-}\int_\T \left(\int_{\{R<|z|<1\}} \left(\frac{1-|z|}{|1-\overline{\z}z|}\right)^{\lambda} \frac{d\mu(z)}{1-|z|}\right)^{\frac{p}{p-q}}\,dm(\z)
\\ &\gtrsim \lim_{R\to 1^-} \int_\T \left(\int_{\Gamma(\z)\setminus\overline{D(0,R)}} \frac{d\mu(z)}{1-|z|}\right)^{\frac{p}{p-q}}\,dm(\z).
\end{split}\end{equation*}
This together with the equivalence
$||f||^p_{H^p}\asymp \int_\T\left(\sup_{z\in\Gamma(\z)}|f(z)|\right)^p\,dm(\z)$ \cite{FefSt}
and
standard arguments, yields that $I_d:H^p\to L^q(\mu)$ is compact.
\end{proof}

\end{document}